\documentclass[a4paper, 10pt]{amsart}

\usepackage{ amssymb, amsmath, amsthm}
\usepackage{graphicx, psfrag, setspace, subfigure, gensymb}
\usepackage{amsfonts}
\usepackage{bbm}
\usepackage{fix-cm}
\usepackage{comment}
\usepackage{tikz,tikz-cd}
\usetikzlibrary{matrix,arrows,decorations.pathmorphing,decorations.pathreplacing}
\tikzset{commutative diagrams/diagrams={baseline=-2.5pt},commutative diagrams/arrow style=tikz}
\usepackage[colorlinks]{hyperref}
\usepackage{float}
\usepackage{setspace} 

\newcommand\C{\mathbb C}

\newcommand{\cE}{\mathcal{E}}

\newcommand{\cK}{\mathcal{K}}

\newcommand{\cO}{\mathcal{O}}

\newcommand{\set}[1]{\left\{{#1}\right\}}
\newcommand{\cone}[1]{\big[{#1}\big]}

\newcommand\into{\hookrightarrow}
\newcommand\To{\longrightarrow}

\newcommand\End{\operatorname{End}}

\newcommand\Ext{\operatorname{Ext}}

\renewcommand\P{\mathbb P}
\newcommand\Gr{\operatorname{Gr}}

\newcommand{\Tot}{\operatorname{Tot}}
\newcommand{\Sym}{\operatorname{Sym}}

\newcommand{\Wedge}{\mbox{\scalebox{1.2}{$\wedge$}}}

\newcommand{\beq}[1]{\begin{equation}\label{#1} }
\newcommand{\eeq}{\end{equation}}
\newcommand{\pgap}{\vspace{5pt}}

\newtheorem{thm}[equation]{Theorem}
\newtheorem{lem}[equation]{Lemma}

\makeatletter \@addtoreset{equation}{section} \makeatother

\let\oldtocsection=\tocsection
\let\oldtocsubsection=\tocsubsection
\let\oldtocsubsubsection=\tocsubsubsection
\renewcommand{\tocsection}[3]{\hspace{0em}\oldtocsection{#1}{#2}{#3}}
\renewcommand{\tocsubsection}[3]{ \hspace{1em} \oldtocsubsection{#1}{\small{#2}}{\small{#3}} }
\renewcommand{\tocsubsubsection}[3]{\hspace{2em}\oldtocsubsubsection{#1}{\small{#2}}{\small{#3}}}

\newcommand{\marginparstretch}{0.6}
\let\oldmarginpar\marginpar
\renewcommand\marginpar[1]{\-\oldmarginpar[\framebox{\setstretch{\marginparstretch}\begin{minipage}{\marginparwidth}{\raggedleft\scriptsize #1}\end{minipage}}]{\framebox{\setstretch{\marginparstretch}\begin{minipage}{\marginparwidth}{\raggedright\scriptsize #1}\end{minipage}}}}

\newcommand{\aand}{\quad\quad\mbox{and}\quad\quad}

\newcommand{\LGr}{\operatorname{LGr}}

\begin{document}

\title{A new 5-fold flop and derived equivalence}
\author{Ed Segal}

\maketitle

\begin{abstract}
We describe a new example of a flop in 5-dimensions, due to Roland Abuaf, with the nice feature that the contracting loci on either side are not isomorphic. We prove that the two sides are derived equivalent.

\pgap
\noindent
\emph{MSC2010 classes: 14E05, 13D09.}
\end{abstract}


\section{Introduction and result}

Let $V$ be a 4-dimensional complex vector space, with a fixed symplectic form. From this data we construct two non-compact Calabi-Yau 5-folds, as follows. Let $\LGr(V)$ denote the Grassmannian of Lagrangian subspaces of $V$, which is a quadric 3-fold, and let $S\subset V$ denote the tautological vector bundle over $\LGr(V)$. Our first 5-fold is the total space of a rank 2 vector bundle:
$$X_+ = \Tot\big( S\otimes \Wedge^2 S \To \LGr(V) \big) $$
Now consider $\P V$, and let $L$ denote the usual tautological line bundle. We let $L^\perp$ denote the rank 3 vector bundle given by the symplectic orthogonal to $L$, and note that it contains $L$ as a subbundle. Our second 5-fold is:
$$X_- = \Tot \big( (L^\perp/L)\otimes L^2 \To \P V\big) $$

\begin{thm}\label{thm.main} $X_+$ and $X_-$ are birational Calabi-Yaus, and they are derived equivalent.
\end{thm}

The fact that both $X_+$ and $X_-$ are Calabi-Yau is a routine calculation, and the fact that they are birational is an elementary piece of geometry (see Section \ref{sect.birationality}). An interesting feature of this example is that the contracting loci on either side are not isomorphic, since one is $\P V$ and the other is $\LGr(V)$. This is in contrast to the case of standard (Atiyah) flops, or Mukai flops.  It is not difficult to construct examples of flops with this feature by using families of standard or Mukai flops, or by allowing one side to be an orbifold, however if one rules out those constructions then we believe that this is the first such example to appear in the literature.

The technical content of the above theorem is the fact that $X_+$ and $X_-$ are derived equivalent. This result is yet another piece of evidence in favour of the well-known conjectures of Bondal--Orlov \cite{BO} and Kawamata \cite{Kaw}, which state that two birational and K-equivalent varieties should be derived equivalent. For this example the argument is not difficult, we simply construct tilting bundles on each space which have the same endomorphism algebra. The proof is given in Sections \ref{sect.derivedeq} and \ref{sect.calculations} below.
\pgap

I had originally hoped to give a proof of this result using the theory of derived categories and Variation-of-GIT, as developed by \cite{BFK, HL} following \cite{Seg}.  However, although it is not difficult to construct GIT problems producing both $X_+$ and $X_-$, I could not find one satisfying the necessary hypotheses.  It would be interesting to see a second proof along those lines.

\subsection{Acknowledgements}
 It is a pleasure to thank Roland Abuaf for allowing me to use his beautiful example, and for asking me if it lead to a derived equivalence. I would also like to thank Dan Halpern-Leistner for useful conversations.

\section{Proof}\label{sect.proof}

\subsection{Birationality}\label{sect.birationality}

As in the previous section we let $V$ be a 4-dimensional vector space with a fixed symplectic form. Let $F$ denote the `isotropic flag variety' 
$$ F = \set{ L\subset S \subset V }$$
where $S\subset V$ is a Lagrangian subspace, and $L\subset S$ is a line. Obviously $F$ is a $\P^1$ bundle over $\LGr(V)$, it is $\P S$. It it also a $\P^1$ bundle over $\P V$, in fact:
$$F \cong \P (L^\perp/L) $$
This is because a plane $S$ contains a line $L$ iff the determinant $\Wedge^2 S \subset \Wedge^2 V$ lies in the subspace $(V/L)\otimes L\subset \Wedge^2 V$, and $S$ is additionally Lagrangian iff $\Wedge^2 S$ lies in $ (L^\perp/L)\otimes L$.

 $F$ carries two obvious line-bundles, $L$ and $\Wedge^2 S$, and we let 
$$\hat{X} = \Tot\big(L\otimes \Wedge^2 S \To F \big)$$
be the total space of their product. There is an evident birational equivalence from $\hat{X}$ to $X_+$ given by `forgetting $L$', at the zero sections this map is the $\P^1$ bundle $F \to \LGr(V)$, and away from the zero sections it is an isomorphism. In fact $\hat{X}$ is the blow-up of $X_+$ along $\LGr(V)$.

 There is also a birational equivalence from $\hat{X}$ to $X_-$ given by the map `forget $S$', \emph{i.e.}~we have an inclusion of bundles over $F$
$$\Wedge^2 S\otimes L \into (L^\perp/L)\otimes L^2 $$
and the latter space is a $\P^1$ bundle over $X_-$, so the composition gives a map from $\hat{X}$ to $X_-$. Again this map is an isomorphism away from the zero sections, and extends the $\P^1$ bundle $F \to \P V$. Also $\hat{X}$ is the blow-up of $X_-$ along $\P V$.

Consequently, $X_+$ and $X_-$ are birationally-equivalent.

\subsection{Derived equivalence}\label{sect.derivedeq}
Kuznetsov \cite{KuzECs} found a full strong exceptional collection on $\LGr(V)$ consisting of the four objects:
$$ \cO, \quad S^\vee, \quad \Wedge^2 S^\vee,\aand (\Wedge^2 S^\vee)^{\otimes 2} $$
(one rank 2 vector bundle and three line-bundles). Pull these four bundles up to $X_+$, and let $T_+$ denote their direct sum. On $\LGr(V)$ these four bundles span the derived category, \emph{i.e.}~there is no non-zero $\cE\in D^b(\LGr(V))$ such that applying $\Ext^\bullet_{\LGr(V)}(-, \cE)$ gives zero on all four bundles. Since the push-down functor from $D^b(X_+)$  to $D^b(\LGr(V))$ has no kernel, it follows from adjunction that $T_+$ spans $D^b(X_+)$. Moreover one can calculate that $T_+$ has no higher self-Ext groups (Lemma \ref{lem.vanishingonX+}) so it is a tilting bundle on $X_+$. Hence $D^b(X_+)$ is equivalent to the derived category $D^b(\End_{X_+}(T_+))$ of the endomorphism algebra of $T_+$ (see for example \cite[Thm.~7.6]{HV}).

Now let $X_o\subset X_\pm$ denote the open set where $X_+$ and $X_-$ are isomorphic, namely the complement of the zero sections, and consider the vector bundle $T_o=T_+|_{X_0}$. We claim that that $T_o$ extends to a vector bundle $T_-$ on $X_-$ which is a tilting bundle. This claim immediately implies the derived equivalence - since $X_+$ and $X_-$ are isomorphic outside of co-dimension two, we have that $\End_{X_-}(T_-)$ is canonically isomorphic to $\End_{X_+}(T_+)$, and hence:
$$D^b(X_+) \cong D^b(\End_{X_+}(T_+))=  D^b(\End_{X_-}(T_-))\cong D^b(X_-)$$
Now we prove the claim. First we extend each of the four summands of $T_0$ to vector bundles on $X_-$. The rank 1 summands are obvious: over $X_o$, the line bundles  $\Wedge^2 S^\vee$ and $L$ are canonically isomorphic, so the necessary line bundles on $X_-$ are:
$$\cO, \quad L,\aand  L^2$$
We also need to extend the rank 2 summand $S^\vee$ from $X_o$ to $X_-$. Note that over $X_o$ we have a short-exact-sequence:
$$S \To L^\perp \oplus S/L \To L^\perp/L$$
This is a just a formal consequence of the inclusions $L\subset S\subset L^\perp$. Now $S/L \cong \Wedge^2 S\otimes L^{-1}\cong L^{-2}$, and the second map in the above sequence can be extended over $X_-$ as the map
$$(1, -\tau) : \;L^\perp \oplus L^{-2} \To L^\perp/L $$
where $\tau$ is the tautological section of $(L^\perp/L)\otimes L^2$. This map has full rank everywhere, so its kernel is a rank 2 vector bundle on $X_-$ which we denote by $\Sigma$. By construction $\Sigma|_{X_o}$ is $S$. Also note that $\Sigma$ fits into a short-exact-sequence
$$ L \To \Sigma \To L^{-2}$$
(in fact it is the unique such non-trivial extension by Lemma \ref{lem.H1onX-}). So $S^\vee$ extends to the bundle $\Sigma^\vee$, which is a non-trivial extension of $L^{-1}$ by $L^2$.

We have extended $T_o$ to a vector bundle $T_-$ on $X_-$, namely:
$$T_- =\cO\oplus  \Sigma^\vee \oplus  L\oplus L^2$$
 It is clear that $T_-$ spans $D^b(X_-)$, since the sub-category  split-generated by $T_-$ contains $\cO, L, L^2$ and $L^{-1}$, and these four line bundles span $D^b(X_-)$ by adjunction. Hence it only remains to show that $T_-$ has no higher self-Ext groups. One can calculate that $H^{>0}(X_-, L^k) = 0$ for $k\leq 2$, that  $H^1(X_-, L^3)=\C$, and that $H^{>1}(X_-, L^3)=0$ (Lemmas \ref{lem.vanishingonX-} and \ref{lem.H1onX-}), from which it follows easily that $\Ext^{>0}_{X_-}(T_-, T_-) = 0$. 

This concludes the proof of Theorem \ref{thm.main}.

\subsection{Cohomology calculations}\label{sect.calculations}

We now give some details of the cohomology calculations required for the argument of the previous section.

\begin{lem}\label{lem.vanishingonGr} Consider the bundle $S^\vee$ on the Grassmannian $\Gr(2, V)$. If  $m\geq -2$, then for any $k\geq 0$ we have:
$$H^{>0}\big(\Gr(2, V),\, \Sym^k S^\vee\otimes (\Wedge^2 S)^{-m}\big) = 0 $$
Setting $m=-3$, we have:
$$H^\bullet\big(\Gr(2, V),\, S^\vee\otimes (\Wedge^2 S)^3\big) =0 \aand   H^\bullet\big(\Gr(2, V),\, (\Wedge^2 S)^3\big) = 0 $$
\end{lem}
\begin{proof} This is a standard Borel--Weil--Bott calculation, see e.g.~\cite{KuzECs}.
\end{proof}

\begin{lem}\label{lem.vanishingonLGr} On the Langrangian Grassmannian $\LGr(V)$, if $m\geq -1$ then for any $k\geq 0$ we have:
$$H^{>0}\big(\LGr(V),\, \Sym^k S^\vee\otimes (\Wedge^2 S)^{-m}\big) = 0 $$
Setting $m=-2$, we have:
$$H^\bullet\big(\LGr(V),\, S^\vee\otimes (\Wedge^2 S)^2\big) =0 \aand  H^\bullet\big( \LGr(V),\, (\Wedge^2 S)^2\big) = 0 $$
\end{lem}
\begin{proof}
$\LGr(V)$ is a linear hyperplane in $\Gr(2,V)$, so we use the short-exact-sequence
$$ \Wedge^2 S \To \cO \To  \cO_{\LGr(V)} $$
on $\Gr(2,V)$ together with Lemma \ref{lem.vanishingonGr}. Alternatively one may compute directly using Borel--Weil--Bott for $Sp(4)$.
\end{proof}

\begin{lem}\label{lem.vanishingonX+} On $X_+$, for the bundle
$$T_+ = \cO \oplus S^\vee \oplus \Wedge^2 S^{-1} \oplus (\Wedge^2 S)^{-2}$$
we have:
$$ \Ext^{>0}_{X_+}( T_+, T_+) = 0 $$
\end{lem}
\begin{proof}
Since $S = S^\vee\otimes \Wedge^2 S$, and 
$$\End(S^\vee) = \cO \;\oplus \;(\Sym^2 S^\vee \otimes \Wedge^2 S) $$
 we need to show that the following bundles have no higher cohomology on $X_+$:
\begin{list}{\roman{enumi})}{\usecounter{enumi}}
\item $\quad (\Wedge^2 S)^{-k}, \quad\quad\quad$ for $k\in [-2,2]$,
\item $\quad S^\vee\otimes (\Wedge^2 S)^{-k}, \quad$ for $k\in [-2,1]$, and
\item $\quad\Sym^2 S^\vee \otimes \Wedge^2 S $.
\end{list}
Pushing down to $\LGr(V)$, this is equivalent to asking that the same bundles have no higher cohomology on $\LGr(V)$ after we tensor them with $\Sym^n S^\vee\otimes (\Wedge^2 S)^{-n}$, for any $n\geq 0$. This follows from Lemma \ref{lem.vanishingonLGr} and the Pieri formula.
\end{proof}

\begin{lem} \label{lem.vanishingonPV} Consider the bundle $V/L$ on $\P V$. For $k\geq 1$ and $m\geq k-1$ we have:
$$H^{>0}\big(\P V,\, \Sym^k(V/L)^\vee\otimes L^{-m}\big) = 0 $$
\end{lem}
\begin{proof}Borel--Weil--Bott.\end{proof}

\begin{lem}\label{lem.vanishingonX-}  On the space $X_-$, for $m\geq -2$  we have:
$$ H^{>0}\big( X_-, \, L^{-m}) = 0 $$
\end{lem}

\begin{proof}
On $\P V $ we have short exact sequence:
$$ (L^\perp/L)\otimes L^2 \To (V/L)\otimes L^2 \To L $$
Hence $X_-$ is a divisor in the space
$$Y = \Tot \big( (V/L)\otimes L^2 \To \P V\big) $$
 cut out by a section of $L$. So we compute $H^\bullet(X_-, L^{-m})$ using the Koszul complex
\beq{eq.KoszulonY} L^{-m-1} \To L^{-m} \To \cO_{X_-}\!\otimes L^{-m} \eeq
on $Y$. Hence it's sufficent to show that $H^{>0}(Y, L^{-m})$ vanishes for $m\geq -2$. Pushing down to $\P V$, we have
$$H^\bullet(Y, L^{-m}) = \bigoplus_{k\geq 0} H^\bullet\big(\P V, \, \Sym^k(V/L)^\vee\otimes L^{-2k-m}\big) $$
and  Lemma \ref{lem.vanishingonPV}, plus the fact that $H^{>0}(\P V, L^{-m})=0$ for $m\geq -3$, ensures that all higher cohomology vanishes.
\end{proof}

\begin{lem}\label{lem.H1onX-} On $X_-$ we have:
$$H^1(X_-, L^3) = \C \aand H^{>1}(X_-, L^3)= 0 $$
\end{lem}
\begin{proof}
By the exact sequence \eqref{eq.KoszulonY} and the fact that $H^{>0}(Y, L^2)=0$, the higher cohomology of $L^3$ is the same on $X_-$ as it is on $Y$. Projecting from $Y$ to $\P V$ and applying Lemma \ref{lem.vanishingonPV}, we see that the only contribution to this higher cohomology is given by
$$H^{>0}\big(\P V, \, (V/L)^\vee\otimes L\big)$$
which is easily calculated.
\end{proof}

\section{A remark on the Fourier-Mukai kernel}

We are still missing a `geometric' construction of the derived equivalence, \emph{i.e.}~a reasonable description of the Fourier-Mukai kernel, but here are a few observations in that direction.

One might reasonably guess that the kernel for our equivalence is supported on the natural geometric correspondence $\hat{X}\subset X_+\times X_-$. Unfortunately this guess is wrong. Our equivalence, considered as a functor $\Phi: D^b(X_-) \to D^b(X_+)$, has the effect
$$\Phi: L^k \mapsto (\Wedge^2 S)^{-k} $$
for $k=0, 1$ or 2, and it sends $L^{-1}$ to the mapping cone
$$\cone{  (\Wedge^2 S)^{-2} \To S^\vee } \quad \in D^b(X_+) $$
on the tautological section. Using the Koszul resolution of the zero section, one sees that this cone is quasi-isomorphic to the ideal sheaf of the zero section, twisted by $\Wedge^2 S$.

Now consider the Fourier-Mukai kernel $\cO_{\hat{X}}$, \emph{i.e.}~the functor `pull-up to $\hat{X}$ and then push down'. By elementary geometric arguments one sees that this functor agrees with $\Phi$ on the objects $\cO$, $L$ and $L^{-1}$. However, it sends $L^2$ to a length 2 complex whose homology in degree zero is $(\Wedge^2 S)^{-2}$, and in degree 1 is the sky-scraper sheaf $\cO_{LGr}$ of the zero section, twisted by $\Wedge^2 S$.

Let us introduce a second kernel:
$$\cK = \cO_{\P V\times \LGr(V)}\otimes L^2\otimes \Wedge^2 S [2] \quad \in  D^b(X_-\times X_+)$$
It's easy to see that the associated functor sends $\cO$, $L$ and $L^{-1}$ to zero, but it sends $L^{-2}$ to the object $\cO_{\LGr}\otimes \Wedge^2 S [-1]$. So the kernel for our equivalence $\Phi$ must be given by the cone on some morphism between $\cO_{\hat{X}}$ and $\cK$. In particular, its support is $\hat{X}\cup \big(\P V \times \LGr(V)\big)$.

\bibliographystyle{halphanum}

\vspace{20pt}
\noindent
\emph{Ed Segal \newline Imperial College London \newline London SW72AZ, UK}

\noindent
\textsf{edward.segal04@imperial.ac.uk}

\end{document}